\documentclass[12pt]{article}
\usepackage{amssymb}
\usepackage{graphicx}
\usepackage{color}
\usepackage{latexsym}
\usepackage{amssymb,amsmath,amstext}
\usepackage{epsfig}
\usepackage{graphics}
\usepackage{amsthm}
\usepackage{verbatim}
\usepackage{placeins}
\usepackage[colorlinks=true, urlcolor=blue, linkcolor=blue, citecolor=blue]{hyperref}
\usepackage[margin=1in]{geometry}
\numberwithin{equation}{section}
\theoremstyle{plain}%
\newtheorem{theorem}{Theorem}
\numberwithin{theorem}{section}
\newtheorem{proposition}[theorem]{Proposition}
\newtheorem{example}[theorem]{Example}
\newtheorem{lemma}[theorem]{Lemma}
\newtheorem{corollary}[theorem]{Corollary}
\newtheorem{definition}[theorem]{Definition}

\DeclareMathOperator{\val}{val}
\newcommand{\C}{\mathbb{C}}

\newcommand{\R}{\mathbb{R}}

\allowdisplaybreaks

\begin{document}

\title{\bf The Tropical Commuting Variety}

\author{Ralph Morrison and Ngoc M. Tran
}
\date{}

\maketitle

\begin{abstract}\noindent 
We study tropical commuting matrices from two viewpoints:  linear algebra and algebraic geometry. 
In classical linear algebra, there exist various criteria to test whether two square matrices commute. We ask for similar criteria in the realm of tropical linear algebra, giving conditions for two tropical matrices that are polytropes to commute. From the algebro-geometric perspective, 
we explicitly compute the tropicalization of the classical variety of commuting matrices in dimension 2 and 3. 
\end{abstract}

\section{Introduction} 

There are various ways to study the pairs of $n\times n$ matrices $X$ and $Y$ over a field $k$ that commute under matrix multiplication.  Linear algebraically, one can ask for criteria to determine commutativity.  Algebro-geometrically, one can study the \emph{commuting variety}, which  is generated by the $n^2$ equations $(XY)_{ij} - (YX)_{ij} = 0$.  These perspectives and many other variants have been studied in the classical setting \cite[\textsection 5]{OCV}.  
This paper considers similar questions for tropical and tropicalized matrices. 

The tropical min-plus algebra $(\overline{\R} , \oplus, \odot)$ is defined by $\overline{\R}=\R\cup\{\infty\}$, $a \oplus b = \min(a,b)$, $a \odot b = a + b$. A pair of $n \times n$ tropical matrices $A = (a_{ij})$, $B = (b_{ij})$ commute if $A \odot B = B \odot A$, where matrix multiplication takes place in the min-plus algebra. Explicitly, this means that for all $1\leq i,j  \leq n$,
$$ \min_{s=1, \ldots, n} a_{is} + b_{sj} = \min_{s=1, \ldots, n} b_{is} + a_{sj}.$$

Tropical linear algebra has extensive applications to discrete events systems \cite{BCOQ}, scheduling \cite{Bu}, pairwise ranking \cite{Tr}, and auction theory \cite{BK}, amongst others. However, the tropical analogues of many fundamental results in classical linear algebra remain open. Commutativity of tropical matrices is one such example. Classically, if $A,B\in \C^{n\times n}$ where $A$ has $n$ distinct eigenvalues, then $A  B=B  A$ if and only if $B$ can be written as a polynomial in $A$ \cite[\textsection 5]{OCV}.  Moreover, if $B$ has $n$ distinct eigenvalues, then $A  B=B  A$ if and only if $A$ and $B$ are simultaneously diagonalizable.  In a similar spirit, we have the following criterion for a special class of matrices called polytropes to commute tropically. Here the Kleene star $A^\ast$ of a polytrope $A$ is the finite geometric sum $A \oplus A^{\odot 2} \oplus \cdots \oplus A^{\odot n}$. 

\begin{theorem}\label{main_theorem}
Suppose $A, B \in \mathbb{R}^{n \times n}$ are polytropes. If $A \oplus B = (A \oplus B)^\ast$, then $A\odot B = B\odot A$. If $A\odot B = B\odot A$, then $(A \oplus B)^{\odot 2} = (A \oplus B)^\ast$. In particular, for $n = 2,3$, $A\odot B=B\odot A$ if and only if $A \oplus B = A \odot B$.
\end{theorem}

Previous works on commuting tropical matrices have also focused on polytropes \cite{KSS, LP}, due to its special role as the projection to the tropical eigenspace \cite{KSS, SSB}. To the best of our knowledge, this is the first necessary and sufficient characterization of commutativity for polytropes for $n < 4$.

The second half of our paper looks at tropical commuting matrices from the viewpoint of tropical algebraic geometry. This is a successful young field bridging combinatorics and algebraic geometry. It has many applications ranging from curve counting, to number theory, to polyhedral geometry, to phylogenetics \cite{MS}. We study the relation between three sets of pairs of $n\times n$ matrices which all `commute tropically' in different sense: the \emph{tropical commuting set} $\mathcal{TS}_n$, the set of all pairs of tropical commuting matrices; the \emph{tropical commuting variety} $\mathcal{TC}_n$, the \emph{tropicalization} of the classical commuting variety; and the \emph{tropical commuting prevariety} $\mathcal{T}_{{\rm pre}, n}$, the intersection of the tropical hypersurfaces corresponding to the natural generators of classical commuting variety. In addition to the inclusion $\mathcal{TC}_n\subset \mathcal{T}_{{\rm pre},n}$ (by definition), one can quickly see that $\mathcal{TC}_n\subset \mathcal{TS}_n$. We show that in general both of these inclusions can be strict, and neither $\mathcal{TS}_n$ nor $\mathcal{T}_{{\rm pre},n}$ is contained in the other when $n>2$. 

\subsection{Outline}

In Section \ref{sec:ts} we present background, notation, and results in tropical linear algebra, and prove Lemma \ref{lem:fixedpoint}.  We use these results to prove Theorem \ref{main_theorem}, and then illustrate the geometry of commuting polytropes.  
In Section \ref{sec:tc} we review basic concepts in tropical algebraic geometry, present the relationships between the spaces $\mathcal{TS}_n$, $\mathcal{TC}_n$, and $\mathcal{T}_{{\rm pre}, n}$, and explicitly compute these sets for $n=2,3$ using the software \textsf{gfan} \cite{Je}. Complete description of our computations, including input files, commands and output files can be found at the public GitHub repository \url{http://github.com/princengoc/tropicalCommutingVariety}. Finally, we conclude with open problems in Section \ref{sec:summary}.

\section{Commuting Polytropes: Algebraic and Geometric Characterizations}\label{sec:ts}

\subsection{Background}

We begin with some notation and basic facts in tropical linear algebra. See \cite[\S 1-3]{Bu} for more details.

If $n$ is a positive integer, let $[n]=\{1,2,.\ldots n\}$. We will  write tropical matrix multiplication as $A\odot B$ to remind the reader of the min-linear nature of the algebra. Let $I$ denote the tropical identity matrix, with $0$ on the diagonal and $\infty$ elsewhere. Let $\mathbb{TP}^{n-1} := \R^n / \R(1, \ldots, 1)$ be the tropical torus. If $C \subset \R^n$ is closed under scalar tropical multiplication, we shall identify it with the set in $\mathbb{TP}^{n-1}$ obtained by normalizing the first coordinate to be 0. The image of a matrix $A$, denoted $\text{im}(A)$, is an example of such a set. 
The \emph{tropical convex hull} between two points $x,y \in \R^n$ is
$$ [x,y] = \{a \odot x \oplus b \odot y: a,b \in \R \}. $$
As a set in $\mathbb{TP}^{n-1}$, this is called the tropical line segment between $x$ and $y$. A \emph{tropical polytope}, also known as tropical semi-module, is the tropical convex hull of finitely many points. The image of an $n \times n$ matrix is a tropical polytope with at most $n$ distinct vertices in $\mathbb{TP}^{n-1}$. For an $n\times n$ matrix $A$ with tropical eigenvalue $0$, the \emph{Kleene star} of $A$ is the matrix $A^\ast = I \oplus \bigoplus_{i=1}^\infty A^{\odot i}$. This is in fact equals to the finite sum $I \oplus \bigoplus_{i=1}^n A^{\odot i}$. 
 
We can view a matrix $A \in \R^{n^2}$ as a map $A:\mathbb{TP}^{n-1} \rightarrow \mathbb{TP}^{n-1}$.  Each of the columns of $A$ defines a point in $\mathbb{TP}^{n-1}$, and the image of $A$ is the tropical convex hull of these points.  A particularly nice case is where the image is full dimensional. This leads us to the following definition.

\begin{definition}
A matrix $A \in \R^{n \times n}$ is a \emph{premetric} if $A_{ii} = 0$, $A_{ij} > 0$ for all $i \neq j \in [n]$.
\end{definition}

In this case, $A$ has eigenvalue $0$, and its image in $\mathbb{TP}^{n-1}$ is a full-dimensional tropical simplex whose main cell has type $(0, 1, \ldots, n-1)$ in the sense of \cite{DS}. 

\begin{definition}  A matrix $A \in \R^{n \times n}$ is a \emph{polytrope} if $A$ is a premetric, and for all $i,j,k \in [n]$, $A_{ij} \leq A_{ik} + A_{kj}$. 
\end{definition}

There are many equivalent characterizations of polytrope, e.g., that it is a premetric and $A = A^{\odot 2}$, or that it is a Kleene star of some matrix \cite[\S 4]{Bu}. A polytrope $A$ has eigenvalue $0$, and the $n$ columns of $A$ are its $n$ eigenvectors. The image of $A$ in $\mathbb{TP}^{n-1}$ is a full-dimensional tropical polytope that is also convex in the usual Euclidean sense. 

For any matrix $A \in \R^{n^2}$ and $b \in im(A)$, we can consider its preimage under $A$, i.e. the set of $x \in \R^n$ such that $A\odot x = b$ \cite[\S 3.1-3.2]{Bu}. If $A$ is a polytrope, this preimage has a simple and explicit form. We note that the following theorem is a special case of Theorem 3.1.1 in \cite{Bu}, attributed to Cunninghame-Green (1960) and Zimmerman (1976). This result was also independently re-discovered by Krivulin \cite{Ki}.   

\begin{theorem}\label{thm:image}
Let $A \in \R^{n^2}$ be a polytrope. Define $I = \{i_1, \ldots, i_k\}$ for some $1 \leq k \leq n$. Suppose that $b \in im(A)$ has the form
\begin{equation}\label{eqn:form.b}
b = \bigoplus_{i \in I} a_i \odot A_i = a_1 \odot A_{i_1} \oplus \ldots \oplus a_k \odot A_{i_k}.
\end{equation}
 Then $A\odot x = b$ if and only if
\begin{equation}\label{eqn:form.x}
x = b + \sum_{j \in [n] \backslash I}t_j\mathsf{e}_j,
\end{equation}
where $t_j \geq 0$, and $\mathsf{e}_j$ is the unit vector on the $j$-th coordinate.
\end{theorem}

The above theorem implies that $A$ is a projection of $\mathbb{TP}^{n-1}$ onto its image, which is the tropical convex hull This is illustrated in Figure \ref{figure:polytrope_map} for a $3\times 3$ polytrope with the columns in $\mathbb{TP}^2$ as dots and their tropical convex hull in grey.  As the matrix acts on $\mathbb{TP}^2$, the three columns of the matrix are fixed, as is their tropical convex hull.  The remainder of the plane except for three rays is divided into three regions that are mapped in the directions $(0,-1)$, $(-1,0)$, or $(1,1)$.  This maps each point to an upside-down tropical line with center at one of the three columns.  The rays of these lines that are not in the tropical convex hull are mapped to the point at the center of the tropical line.  
 \begin{figure}[h]
\centering
\includegraphics[scale=0.7]{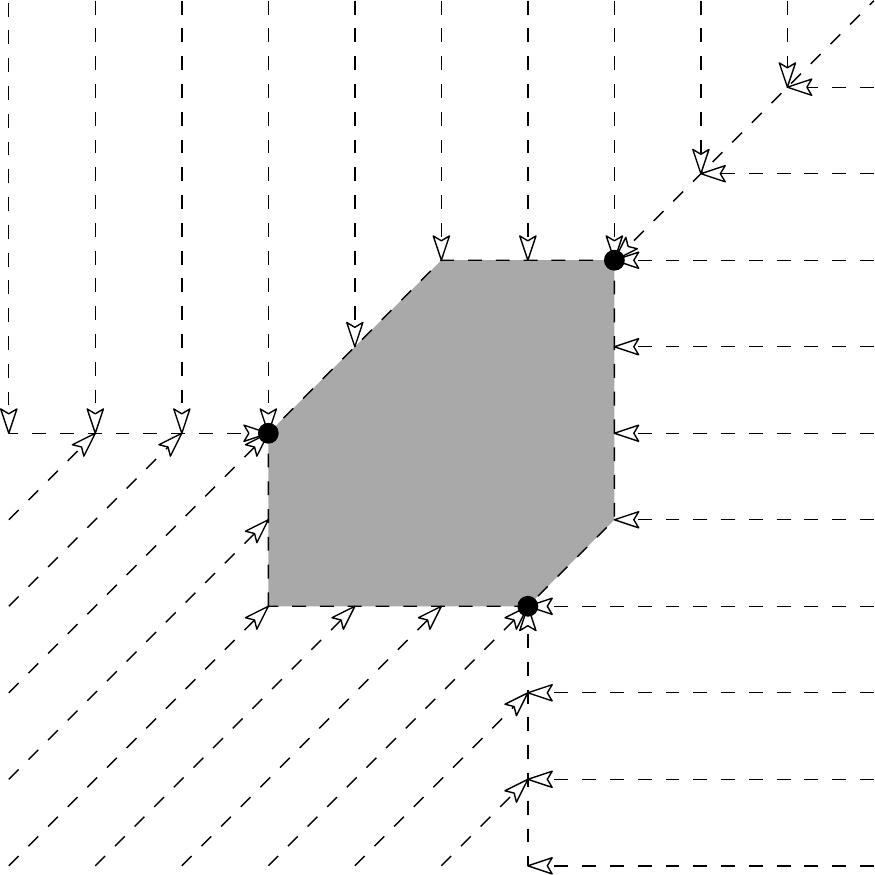}
\label{figure:polytrope_map}
\caption{The action of a polytrope on $\mathbb{TP}^2$.}
\end{figure}

Note that $A\odot B = B\odot A$ means for each $i = 1, \ldots, n$, the projection of the $i$-th column of $B$ onto the image of $A$ equals the projection of the $i$-th column of $A$ onto the image of $B$. Thus, Theorem \ref{thm:image} gives an easy geometric check if two polytropes commute. We now give explicit examples for $n = 3$,  using three polytropes whose images are illustrated in Figure \ref{figure:polytropes}.  Let $A$, $B$, and $C$ have the images of their columns labelled by dots, boxes, and crosses, respectively. 

\begin{figure}[h]
\centering
\includegraphics[width=0.37\textwidth]{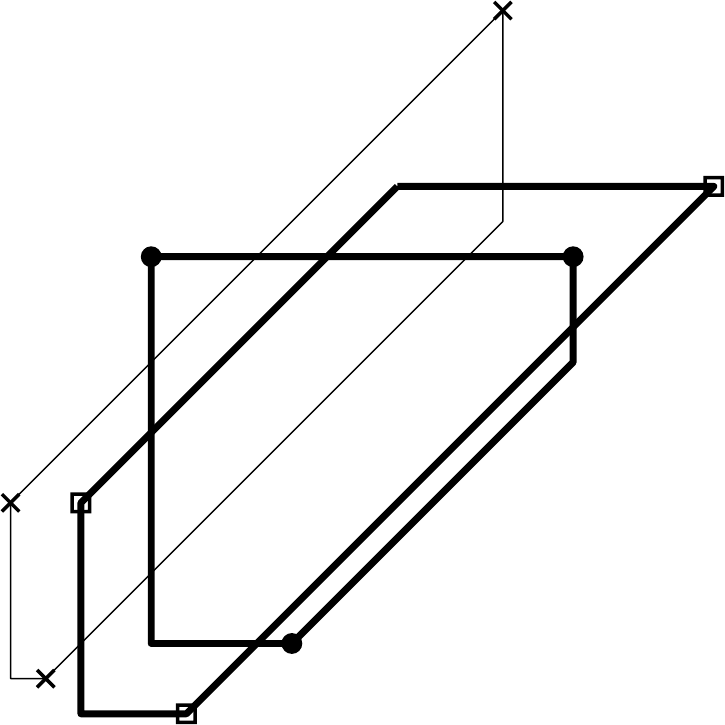}
\label{figure:polytropes}
\caption{The images of three polytropes for $n = 3$}
\end{figure}

\begin{example}\rm{The matrices $A$ and $B$ (in bold) commute.  Consider $\text{im}(A)\cap \text{im}(B)$, which is a hexagon.  The vertices of the hexagon are the vertices of $\text{im}(A \odot B)$; to see this, simply map the columns of $B$ to $\text{im}(A)$ in the natural way.  Similarly, the vertices of the hexagon are the vertices of $\text{im}(B \odot A)$.  It follows that $A \odot B=B \odot A$.

}
\end{example}

\begin{example}\rm{The matrices $A$ and $C$ do not commute.  The pentagon $\text{im}(A)\cap \text{im}(C)$ is \emph{not} $\text{im}(A \odot C)$ (or $\text{im}(C \odot A)$).  For instance, the upper-right cross vertex is not mapped to this intersection by the action of $A$.  This means that $A$ and $C$ do not commute.

}
\end{example}

We close this section by collecting some useful facts about premetrics. Only the last two statements are new, and they are needed for the proof of Theorem \ref{main_theorem}. Therefore, we only prove those statements.
\begin{lemma}\label{lem:fixedpoint}
If $A, B \in \mathbb{R}^{n \times n}$ are premetrics, then the following hold:
\begin{enumerate}
	\item $A\odot B \leq A \oplus B$.
	\item $A^{\odot (n-1)} = A^\ast$.
	\item $A^{\odot 2} = A$ if and only if $A = A^\ast$. 
	\item $A\odot x = x$ if and only if $x$ is in the image of $A^\ast$. 
	\item $\text{im}((A\odot B)^\ast) = \text{im}((A \oplus B)^\ast) = \text{im}(A^\ast) \cap \text{im}(B^\ast)$.
	\item $(A\odot B)^\ast = (B\odot A)^\ast = (A \oplus B)^\ast.$
\end{enumerate}
\end{lemma}

\begin{proof}
The last statement is the matrix multiplication version of the second last, so let us prove the later. By one characterization of the Kleene star, \cite{JK}
\begin{align*}
\text{im}(A^\ast) &= \{x \in \mathbb{TP}^{n-1}: x_i - x_j \leq A_{ij}\} \\
\text{im}(B^\ast) &= \{x \in \mathbb{TP}^{n-1}: x_i - x_j \leq B_{ij}\}.
\end{align*}
This implies
$$ \text{im}(A^\ast) \cap \text{im}(B^\ast) = \{x \in \mathbb{TP}^{n-1}: x_i - x_j \leq \min(A_{ij}, B_{ij})\} = \text{im}((A\oplus B)^\ast). $$
Consider the first equality, that is, the claim that $\text{im}((A\odot B)^\ast) = \text{im}((A \oplus B)^\ast)$. As before,
$$ \text{im}((A\odot B)^\ast) = \{x \in \mathbb{TP}^{n-1}: x_i - x_j \leq (A\odot B)_{ij}\}  $$
Now, $(A\odot B)_{ij} = \min_k A_{ik} + B_{kj} = \min\{A_{ij}, B_{ij}, \min_{k \neq i,j} A_{ik} + B_{kj}\}.$ Thus
$$ \text{im}((A\odot B)^\ast) \subseteq \text{im}((A\oplus B)^\ast).$$
Conversely, suppose that $x \in \text{im}(A^\ast) \cap \text{im}(B^\ast)$. By the fourth statement of the lemma,
$$ A\odot B\odot x = B\odot A\odot x = x, $$
therefore $x \in \text{im}((A\odot B)^\ast).$ So $\text{im}(A^\ast) \cap \text{im}(B^\ast) \subseteq \text{im}((A\odot B)^\ast)$. This proves the desired equality.
\end{proof}

\subsection{Proof of the main theorem}

With the results from the previous section, we are now ready to prove our main result.

\begin{proof}[Proof of Theorem \ref{main_theorem}] $\,$ \\
Suppose that $A \oplus B = (A \oplus B)^\ast$. By Lemma \ref{lem:fixedpoint}, $A \oplus B = (A \oplus B)^2$. We have
$$ A \oplus B = A^{\odot 2} \oplus B^{\odot 2} \oplus A\odot B \oplus B \odot A = A \oplus B \oplus A\odot B \oplus B \odot A.$$
This implies $A \oplus B \leq A\odot B, B\odot A$. By Lemma \ref{lem:fixedpoint}, $A\odot B, B\odot A \leq A \oplus B$. So we must have
$$ A\odot B, B\odot A = A \oplus B, $$
which then implies $A\odot B = B\odot A$. Now, suppose that $A\odot B = B\odot A$. For any $m \geq 2$,C
$$ (A \oplus B)^{\odot m} = \bigoplus_{k=1}^m A^{\odot k}\odot B^{m-k} = \bigoplus_{k=1}^mA \odot B = A \oplus B \oplus A\odot B = A\odot B $$
Therefore, $ (A \oplus B)^{\odot 2} = (A \oplus B)^{\ast}.$
\end{proof}

\begin{corollary}\label{corollary:three}
For $n = 3$, $A\odot B = B\odot A$ if and only if $A \oplus B = (A \oplus B)^\ast$.
\end{corollary}
\begin{proof}
The theorem supplies the ``if'' direction. For the converse, note that $A \odot B = B \odot A$ implies
$$ (A \oplus B)^2 = A \oplus B \oplus A\odot B = A\odot B. $$
Now, suppose for the sake of contradiction that $A\odot B$ is strictly smaller than $A \oplus B$ at some coordinate, say, $(1,2)$. That is,
$$ (A\odot B)_{12} = \min \{A_{11} + B_{12}, A_{12} + B_{22}, A_{13} + B_{32}\} $$
But $A$ and $B$ have zero diagonals, and so
$$ (A\odot B)_{12} = \min \{B_{12}, A_{12}, A_{13} + B_{32}\}.$$
For strict inequality to occur, we necessarily have $(A\odot B)_{12} = A_{13} + B_{32}$. But $A$ and $B$ are polytropes, so 
$$ A_{12} \leq A_{13} + A_{32}, B_{12} \leq B_{13} + B_{32}.$$
Therefore, 
$$ A_{32} > B_{32}, \hspace{1em} A_{13} < B_{13}. $$
On the other hand, $(A\odot B)_{12} = (B\odot A)_{12}$, and by the same argument, we necessarily have
$$ (B\odot A)_{12} = B_{13} + A_{32} < B_{12} < B_{13 } + B_{32},$$
which implies $A_{32} < B_{32}$, a contradiction. Hence there is no coordinate $(i,j) \in [3]\times[3]$ such that $A \odot B_{ij} < A_{ij} \oplus B_{ij}$. In other words, $A\odot B = A \oplus B$, which then implies $A \oplus B = (A \oplus B)^2$.
\end{proof}

Theorem \ref{main_theorem} implies the set inclusion 
$$ \{(A,B): A \oplus B = (A \oplus B)^\ast\} \subseteq \{(A,B): A\odot B = B\odot A\} \subseteq \{(A,B): (A \oplus B)^2 = (A \oplus B)^\ast\}. $$
For $n = 3$, the corollary implies
$$ \{(A,B): A \oplus B = (A \oplus B)^\ast\} = \{(A,B): A\odot B = B\odot A\} \subset \{(A,B): (A \oplus B)^2 = (A \oplus B)^\ast\} = \R^{2n^2}.$$

These inclusions are strict for $n \geq 4$. Consider the following two examples for $n = 4$.


\begin{example} \rm{ [$A\odot B = B\odot A$ but $A \oplus B > (A \oplus B)^\ast$]

Let
$$ A = 
\left[
\begin{array}{cccc}
0.00 & 4.10 & 3.43 & 0.95 \\
4.94 & 0.00 & 1.20 & 5.89 \\
3.74 & 4.44 & 0.00 & 4.69 \\
3.39 & 6.92 & 2.48 & 0.00
\end{array} \right], \hspace{1em}
B = 
\left[
\begin{array}{cccc}
0.00 & 1.11 & 8.21 & 9.02 \\
6.74 & 0.00 & 7.61 & 9.82 \\
9.96 & 9.56 & 0.00 & 9.77 \\
1.03 & 2.14 & 1.36 & 0.00
\end{array}\right] .
 $$
One can check that $A\odot B = B\odot A$, but $A \oplus B$ differs from $(A \oplus B)^2$ in the $(1,2)$ entry: 
$$ (A \oplus B)_{13} = 3.43 > (A\oplus B)^2_{13} = 2.31.$$

}
\end{example}

\begin{example}[$(A \oplus B)^2 = (A \oplus B)^\ast$ but $A\odot B \neq B\odot A$]
Let
$$ A = 
\left[
\begin{array}{cccc}
0.00 & 1.09 & 4.02 & 3.33 \\
6.77 & 0.00 & 2.93 & 3.47 \\
7.77 & 8.00 & 0.00 & 6.20 \\
3.30 & 1.85 & 1.39 & 0.00
\end{array} \right], \hspace{1em}
B = 
\left[
\begin{array}{cccc}
0.00 & 5.02 & 1.45 & 2.58 \\
3.53 & 0.00 & 2.01 & 2.12 \\
7.10 & 3.57 & 0.00 & 1.13 \\
7.71 & 6.04 & 2.47 & 0.00
\end{array}\right] .
 $$

\end{example}

The following is an example for $n = 3$ that shows that it is not sufficient to have $A\odot B = A \oplus B$: one needs $A \oplus B = (A\odot B) \oplus (B\odot A)$ for $A$ and $B$ to commute. 

\begin{example}\rm{
$$ A = 
\left[
\begin{array}{ccc}
0.00 &  6.4 & 6.10 \\
3.01 &  0.0 & 0.54 \\
5.41 &  2.4 & 0.00
\end{array} \right], \hspace{1em}
B = 
\left[
\begin{array}{ccc}
0.00 & 2.25 & 5.04 \\
6.81 & 0.00 & 2.79 \\
4.02 & 6.27 & 0.00
\end{array}\right] .
 $$

In this case, $A\odot B = A \oplus B$, but $B\odot A \neq A \oplus B$. These two matrices differ in the $(1,3)$ coordinate
$$ (B\odot A)_{13} = 2.79 < (A \oplus B)_{13} = 5.04, $$
so in particular, $A\odot B \neq B\odot A$.
}
\end{example}

\section{Tropicalization of the classical commuting variety}\label{sec:tc}

Let $k$  be an algebraically closed non-Archimedean field with non-trivial valuation, such as the Puiseux series over $\mathbb{C}$, and fix an integer $n\geq 2$.  Let $S_n=k[\{x_{ij},y_{ij}\}_{i,j\in\{1,\ldots,n\}}]$ and let $I_n\subset S$ be the ideal generated by the $n^2$ elements of the form
\begin{equation}
\label{classical_generators}
\sum_{i=1}^nx_{ik}y_{kj}-\sum_{j=1}^nx_{\ell j}y_{j\ell}
\end{equation}
where $k,\ell\in \{1,\ldots,n\}$.  We call the variety $V(I_n)$ the $n\times n$ \emph{commuting variety} over $k$.  It is irreducible and has dimension $n^2+n$ \cite{GS, MT}.  Its classical points correspond to pairs of matrices $X,Y\in k^{n\times n}$ that commute.  
Since $k$ is algebraically closed, we may identify the variety with pairs of commuting matrices.  The situation is more subtle tropically. As in the introduction, we consider three tropical spaces:

\begin{itemize}

\item The  \emph{tropical commuting set} $\mathcal{TS}_n$, which is the collection of all pairs of $n\times n$  tropical commuting matrices in $\R^{2n^2}$.

\item  The \emph{tropical commuting variety} $\mathcal{TC}_n$, which is the tropicalization of the commuting variety.

\item The \emph{tropical commuting prevariety} $\mathcal{T}_{{\rm pre}, n}$, which is the tropical prevariety defined by the $n^2$ equations in (\ref{classical_generators}).

\end{itemize}

The tropicalization of a variety over such a field can be defined as the Euclidean closure of the image of the variety under coordinate-wise valuation.  
For completeness we recall an alternate definition of a tropical variety:  for $\omega = (\omega^x, \omega^y) \in \R^{n^2} \times \R^{n^2}$, $f \in \R[x_{ij}, y_{ij}]$, let $in_\omega(f)$ denote the initial form of $f$, $in_w(I_n) := \langle in_w(f): f \in I_n \rangle$ the initial ideal of $I_n$. The tropical variety $\mathcal{T}(I_n)$ is the subcomplex of the Gr\"{o}bner fan of $I_n$ consisting of cones $C_\omega$ where $in_\omega(I_n)$ does not contain a monomial. 

Our first result concerns the homogeneity space of $\mathcal{T}(I_n)$, denoted $\mathrm{homog}(I_n)$ This is the set of $\omega \in \R^{2n^2}$ such that $in_\omega(I_n) = I_n$. In our case, this set is a subspace of dimension $n+1$, which coincides with the lineality space of the Gr\"{o}bner fan of $I_n$. 

\begin{proposition}
Suppose $n \geq 3$. For $\omega = (\omega^x, \omega^y) \in \R^{n^2} \times \R^{n^2}$, $\omega \in \mathrm{homog}(I_n)$ if and only if there exists $a, b \in \R$ and $c \in \R^n$ such that for all $i,j \in [n]$
\begin{equation}\label{eqn:homog}
\omega^x_{ii} = a, \omega^y_{jj} = b, \omega^y_{ij} = \omega^x_{ij} -a+b, \mbox{ and } \omega^x_{ij} = c_i - c_j + a.
\end{equation}
In particular, $\mathrm{homog}(I_n)$ has dimension $n+1$ for $n \geq 3$. For $n = 2$, $\omega \in \mathrm{homog}(I_2)$ if and only if there exist $a,b \in \R$ such that
\begin{equation}\label{eqn:homog2}
\omega^x_{11} = \omega^x_{22} = a, \omega^y_{11} = \omega^y_{22} = b, \omega^y_{12} = \omega^x_{12} - a + b, \mbox{ and } \omega^y_{21} = \omega^x_{21} - a + b.
\end{equation}
In particular, $\mathrm{homog}(I_2)$ has dimension $4$.
\end{proposition}
\begin{proof}
We shall prove that $\omega$ satisfies (\ref{eqn:homog}) if and only if $in_\omega(g_{ij}) = g_{ij}$ for all $i,j \in [n]$. Since the $g_{ij}$'s generate $I_n$, this then implies $in_\omega(I_n) = in(I_n)$. 

Suppose $\omega$ is such that $in_\omega(g_{ij}) = g_{ij}$. For each fixed $i,j \in [n]$, the monomials $x_{ii}y_{ij}$ and $y_{ij}x_{jj}$ have equal weights. Thus $\omega^x_{ii} = \omega^x_{jj} = a$ for all $i,j \in [n]$. Similarly, $\omega^y_{ii} = \omega^y_{jj} = b$. Now, $x_{ii}y_{ij}$ and $x_{ij}y_{jj}$ have equal weights. Thus
\begin{equation}\label{eqn:equal.weights1}
\omega^y_{ij} = \omega^x_{ij} - \omega^x_{ii} + \omega^y_{jj} = \omega^x_{ij} - a + b
\end{equation}
Consider a triple $i,j,k \in [n]$ of distinct indices. The monomials $x_{ik}y_{kj}$ and $x_{ij}y_{jj}$ have equal weights. Thus
\begin{equation}
0 = \omega^x_{ik} + \omega^y_{kj} - (\omega^x_{ij} + b) = \omega^x_{ik} + (\omega^x_{kj} - a + b) - (\omega^x_{ij} + b) = \omega^x_{ik} + \omega^x_{kj} - a - \omega^x_{ij}.
\end{equation}
Since this holds for all triples $i,j,k \in [n]$, we necessarily have 
\begin{equation}\label{eqn:equal.weights2}
\omega^x_{ij} = c_i - c_j + a
\end{equation} 
for some $c \in \R^n$. Thus, $\omega$ is of the form given in (\ref{eqn:homog}).

Finally, for $n = 2$, (\ref{eqn:equal.weights1}) still holds. So we have (\ref{eqn:homog2}). 
Define $\omega^x_{12} = c$, $\omega^x_{21} = d$, we see that $\mathrm{homog}(I_2)$ is a linear subspace of $\R^8$ of dimension 4, parametrized by four parameters $a,b,c,d$.
\end{proof}

\subsection{The $2\times 2$ Tropical Commuting Variety}

The tropical variety $\mathcal{TC}_2$ lives in an $8$-dimensional ambient space, corresponding to the four $x_{ij}$ and the four $y_{ij}$ coordinates.  It is $6$-dimensional, with a $4$-dimensional lineality space.  Modding out by this lineality space gives a $2$-dimensional fan with f-vector 
$\left(\begin{matrix}1& 4 & 6\end{matrix}\right),$
meaning there are four rays and six $2$-dimensional cones.  The tropical variety is simplicial and pure.

Computation with \textsf{gfan} shows that the tropical prevariety equals the tropical variety. In other words, the following three polynomials
\begin{align*}
g_{11} &= x_{12}y_{21} - y_{12}x_{21}, \\
g_{12} &= x_{11}y_{12} + x_{12}y_{22} - y_{11}x_{12} - y_{12}x_{22}, \mbox{ and } \\
g_{21} &= x_{21}y_{11} + x_{22}y_{21} - y_{21}x_{11} - y_{22}x_{21}
\end{align*}
form a tropical basis for $\mathcal{T}(I_2)$.   We summarize this and slightly more in the following proposition.

\begin{proposition}\label{proposition_2x2}
We have $\mathcal{T}_{{\rm pre},2} = \mathcal{TC}_2 = \mathcal{TS}_2 \cap \{a_{12}+b_{21}=a_{21}+b_{12}\}$.
The homogeneity space is
$$ \omega^x_{11} = \omega^x_{22} = a, \omega^y_{11} = \omega^y_{22} = b, \omega^y_{12} = \omega^x_{12} - a + b, \mbox{ and } \omega^y_{21} = \omega^x_{21} - a + b.$$
\end{proposition}

\begin{proof}
We have proven everything except the relationship between $\mathcal{TS}_2$ and the other spaces. If $(A,B)\in   \mathcal{TS}_2$, then two of the generators of our tropical basis, namely $g_{12}$ and $g_{21}$, are tropically satisfied.  The final generator $g_{11}$. is tropically satisfied if and only $a_{12}+b_{21}=a_{21}+b_{12}$, giving the claimed equality.
\end{proof}

\begin{example} \rm{ Let $k=\C\{\!\{\!t\!\}\!\}$ be the field of Puiseux series over $\C$ with the usual valuation.  Proposition \ref{proposition_2x2} tells us when commuting  $2\times 2$ tropical matrices with entries in $\val(k)$ can be lifted to commuting tropical matrices in $k$.  Since the pair of matrices \[\left(\left(
\begin{matrix} 
0 & 4
\\ 2 & 0
\end{matrix}
\right),
\left(
\begin{matrix} 
0 & 3
\\ 1 & -1
\end{matrix}
\right)\right)\] satisfies $4+1=2+3$, so they can be lifted, for instance to the pair of matrices
\[\left(\left(
\begin{matrix} 
1+t & t^4
\\ t^2 & 2
\end{matrix}
\right),
\left(
\begin{matrix} 
1 & t^3
\\ t & t^{-1}
\end{matrix}
\right)\right).\]
}
\end{example}

\begin{figure}[h]
\centering
\includegraphics[scale=1.1]{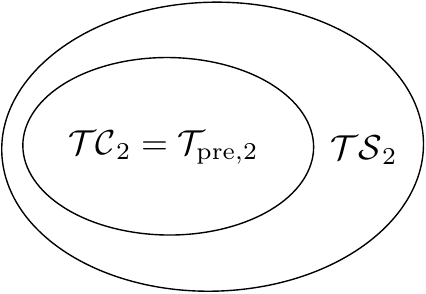}
\label{figure:three_spaces_2}
\caption{The three spaces for $n=2$.}
\end{figure}

The relationship between the three spaces for $n=2$ is illustrated in Figure \ref{figure:three_spaces_2}.  We now give an example to demonstrate that the containment $\mathcal{T}_{{\rm pre},2}\subset \mathcal{TS}_2$ really is proper.

\begin{example}\rm{
Consider the pair of matrices $\left(\left(
\begin{matrix} 
0 & 2
\\ 1 & 0
\end{matrix}
\right),
\left(
\begin{matrix} 
0 & 1
\\ 1 & 0
\end{matrix}
\right)\right)$.  These commute under tropical matrix multiplication, but do not tropically satisfy the polynomial $g_{11}=x_{12}y_{21}-y_{12}x_{21}$. Thus, this pair of matrices is in $\mathcal{TS}_2$, but not in $ \mathcal{T}_{{\rm pre},2}$.
}
\end{example}

\subsection{The $3\times 3$ Tropical Commuting Variety}  

 For higher dimensions, the containment relation between the three sets is as pictured in Figure \ref{figure:three_spaces_n}. We state and prove the result for $n = 3$. The proofs for cases with $n > 3$ are similar.

\begin{proposition}\label{prop:containment}
We have $\mathcal{TC}_3 \subsetneq \mathcal{T}_{{\rm pre},3} \cap \mathcal{TS}_3$, and neither $\mathcal{T}_{{\rm pre},3}$ nor $\mathcal{TS}_3$ are contained in one another.
\end{proposition}

\begin{proof} To see that each region in Figure \ref{figure:three_spaces_n} is really nonempty, consider the following examples.
\begin{itemize}
\item[(a)]  The pair of matrices $A = \left[\begin{array}{ccc}
0 &  2 &   0 \\
2 & 0 &  8 \\
0 &  4 &   0
\end{array}\right], \hspace{1em}
B = \left[\begin{array}{ccc}
12 &  0 &   1 \\
0 & 2 & 0 \\
1 &  0 &  6
\end{array}\right]$ is in $ \left(\mathcal{T}_{{\rm pre},3}\cap \mathcal{TS}_3\right)\setminus \mathcal{TC}_3$.
Indeed, direct computation shows that $(A,B)\in\mathcal{T}_{{\rm pre},3}\cap \mathcal{TS}_3$.  Computations with \textsf{gfan} show that the initial monomial ideal with this weight vector contains the monomial $x_{31}y_{12}y_{31}y_{21}$. Thus, $(A,B)$ does not lie in the tropical variety. The polynomial with this leading term is given by
\begin{equation}\label{eqn:bad.polynomial}
(XY-YX)_{31}y_{32}y_{21} - (XY-YX)_{32}y_{31}y_{21} - (XY-YX)_{21}y_{31}y_{32}.
\end{equation}
Each of the three terms $(XY-YX)_{31}$, $(XY-YX)_{32}$ and $(XY-YX)_{21}$ is a sum of six monomials, two of which are initial monomials. This gives 18 monomials in total with 6 initial monomials. However, the six initial monomials come in three pairs, which are cancelled out by the signs. So (\ref{eqn:bad.polynomial}) has 12 monomials, and the weights are such that there is a unique leading term.

\item[(b)]  The pair of matrices  $C = \left[\begin{array}{ccc}
0 &  1 &   4 \\
1 & 0 &  4 \\
4 &  4 &   0
\end{array}\right], \hspace{1em}
D = \left[\begin{array}{ccc}
0 &  2 &   4 \\
1 & 0 & 4 \\
4 &  4 &  0
\end{array}\right]$ is in $\mathcal{TS}_3\setminus\mathcal{T}_{{\rm pre},3}$.

Indeed, direct computation shows that these matrices commute, and that containment in $\mathcal{T}_{{\rm pre},3}$ fails on the $(1,1)$ and the $(2,2)$ entries of the products.

\item[(c)] The pair of matrices   $E = \left[\begin{array}{ccc}
0 &  1 &   0 \\
3 & 0 &  1 \\
0 &  3 &   0
\end{array}\right], \hspace{1em}
F = \left[\begin{array}{ccc}
1 &  0 &   3 \\
0 & 1 & 0 \\
1 &  0 &  3
\end{array}\right]$ is in $\mathcal{T}_{{\rm pre},3}\setminus \mathcal{TS}_3$.

Indeed, direct computation shows that these matrices fail to commute in the $(3,3)$ entry of the products, and that $(E,F)\in \mathcal{T}_{{\rm pre},3}$.
\end{itemize}

In summary, we have $(A,B)\in \left(\mathcal{T}_{{\rm pre},3}\cap \mathcal{TS}_3\right)\setminus \mathcal{TC}_3$, $(C,D)\in  \mathcal{TS}_3\setminus\mathcal{T}_{{\rm pre},3}$, and $(E,F)\in \mathcal{T}_{{\rm pre},3}\setminus \mathcal{TS}_3$.
\end{proof}

\begin{figure}[h]
\centering
\includegraphics[scale=1.2]{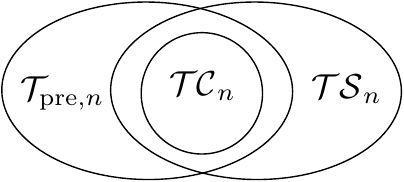}
\label{figure:three_spaces_n}
\caption{The three spaces for $n>2$.}
\end{figure}
 
\subsubsection{The geometry of $\mathcal{TC}_3$ and $\mathcal{T}_{{\rm pre},3}$}

The tropical variety $\mathcal{TC}_3$ lives in an $18$-dimensional ambient space, corresponding to the nine $x_{ij}$ and the nine $y_{ij}$ coordinates.  It is $12$-dimensional, with a $4$-dimensional lineality space. Modding out gives us an $8$-dimensional space. The f-vector is
$$\left(\begin{matrix}1& 1658 &23755& 143852& 481835 &972387& 1186489& 808218 &235038\end{matrix}\right),$$
which ranges from the 1658 rays to the 235,038 $8$-dimensional cones.  The tropical variety is pure, but not simplicial.

The tropical prevariety is much bigger than the tropical variety. The prevariety is neither pure nor simplicial. Modulo the lineality space, its largest cones are of dimension 10. Its $f$-vector is 
$$\left(\begin{matrix}1& 146& 2290& 16322& 66193& 162886& 241476& 199030& 71766& 2397&58\end{matrix}\right).$$

As shown in the proof of Proposition \ref{prop:containment}, apart from the generators of the pre-variety, the tropical basis for $\mathcal{TC}_3$ necessary contains the polynomial
$$ (XY-YX)_{31}Y_{32}Y_{21} - (XY-YX)_{32}Y_{31}Y_{21} - (XY-YX)_{21}Y_{31}Y_{32} $$
and all of its permutations under $\mathbb{S}_3 \times \mathbb{S}_2$, by permuting the rows and columns of the matrices simultaneously, and swapping $X$ and $Y$. By a similar argument, another set of polynomials in the tropical basis are all permutations of
$$(XY-YX)_{12}Y_{21} - (XY-YX)_{21}Y_{12}.$$
However, these two sets of polynomials alone cannot account for the gap in the dimension of the maximal cones between $\mathcal{TC}_3$ and $\mathcal{T}_{{\rm pre},3}$. We suspect that the full tropical basis of $\mathcal{TC}_3$ contains many more polynomials. Computing this basis explicitly is an interesting open question. 

\subsubsection{The symmetric commuting pre-variety}
As a first step to computing the tropical basis of $\mathcal{TC}_3$, we study the analogue of $\mathcal{T}_{{\rm pre},3}$ and $\mathcal{TC}_3$ for pairs of commuting \emph{symmetric matrices}, so that $X = X^T$ and  $Y = Y^T$. These live in a $12$-dimensional ambient space, corresponding to the six $x_{ij}$ and the six $y_{ij}$ coordinates. The ideal $I^{\rm sym}_3$ is generated by the following three polynomials:
\begin{align*}
(XY)_{12} - (YX)_{12} &= x_{11}y_{12}-y_{11}x_{12}+x_{12}y_{22}-y_{12}x_{22}+x_{13}y_{23}-y_{13}x_{23} \\
(XY)_{13} - (YX)_{13} &= x_{11}y_{13}-y_{11}x_{13}+x_{12}y_{23}-y_{12}x_{23}+x_{13}y_{33}-y_{13}x_{33} \\
(XY)_{23} - (YX)_{23} &= x_{12}y_{13}-y_{12}x_{13}+x_{22}y_{23}-y_{22}x_{23}+x_{23}y_{33}-y_{23}x_{33}.
\end{align*}

The symmetric tropical commuting variety is $9$-dimensional, with a $2$-dimensional lineality space. Its f-vector is
$$\left(\begin{matrix}1 &66 & 705 & 3246 & 7932 & 10878 & 8184 & 2745\end{matrix}\right).$$
The symmetric tropical commuting prevariety is only one dimension bigger. It has dimension 10, also with a $2$-dimensional lineality space. Its $f$-vector is
$$\left(\begin{matrix}
1 & 39 & 375 & 1716 & 4359 & 6366 & 5136 & 1869 & 6
\end{matrix}\right).$$
Under the action of $\mathbb{S}_3 \times \mathbb{S}_2$, the six cones of dimension ten form three orbits. We name them type I, II and III. Type I has orbit size 1, with initial monomials 
$$ x_{13}y_{23} - x_{23}y_{13}, \hspace{1em} x_{12}y_{23} - x_{23}y_{12}, \hspace{1em} x_{12}y_{13} - x_{13}y_{12}.  $$
Type II has orbit size 2, with initial monomials
$$ x_{12}y_{11} - x_{12}y_{22}, \hspace{1em} x_{13}y_{11} - x_{13}y_{33}, \hspace{1em} x_{23}y_{22} - x_{23}y_{33}.  $$
Type III has orbit size 3, with initial monomials
$$ x_{11}y_{12} - x_{12}y_{11}, \hspace{1em} x_{11}y_{13} - x_{13}y_{11}, \hspace{1em} x_{12}y_{13} - x_{13}y_{12}. $$
In theory, since there are three generators with six terms, there can be at most $\binom{6}{2}^3 = 15^3$ possible cones of the symmetric tropical commuting prevariety with maximal dimension. It remains to be understood why only the above six cones are full-dimensional. 

\section{Summary and Future Directions}\label{sec:summary}

In this work we studied tropical commuting matrices from the perspectives of linear algebra and algebraic geometry. We gave algebraic and geometric conditions for $n \times n$ polytropes, a special class of matrices, to commute. Our conditions are necessary and sufficient for $n = 2,3$. We also tropicalize the classical commuting variety of $n \times n$ matrices, explicitly compute them for $n = 2,3$, and study their relations to the tropical commuting prevariety and the tropical commuting set. Two major open problems remain in dimensions $n \geq 3$: to find a complete characterization of the tropical commuting set, and to find a formula for the tropical basis for the tropical commuting variety. 

Another future direction is to consider triples of pairwise-commuting $n\times n$ matrices.  It was shown in \cite{Ge,GS} that the variety of triples of commuting $n\times n$ matrices is irreducible for $n\leq 4$ but reducible for $n\geq 32$.  More generally, one can study the space $\mathcal{C}(d,n)$ of commuting $d$-tuples of $n\times n$ matrices.  For $d\geq 4$ and $n\geq 4$, this variety is reducible \cite{Ge}.  Studying the tropical analogues of these spaces would be a natural generalization of the work we have done here.

\medskip
{\bf Acknowledgements.}\\
We thank JM Landsberg, Laura Matusevich, and Bernd Sturmfels for guidance and advice throughout this project. Ralph Morrison was supported by the US National Science Foundation.  Ngoc Tran was supported by the Simons Foundation ($\#197982$ to The University of Texas at Austin).


\begin{thebibliography}{10}

\setlength{\itemsep}{-1mm}


\bibitem[AD]{AD} S. Abeasis and A. Del Fra: \emph{Degenerations for the representations of a quiver of type $A_m$}, J. Algebra, 93 (1985), 376-412.
\label{AD}

\bibitem[ADK]{ADK}  S. Abeasis, A. Del Fra and H. Kraft: \emph{The geometry of the representations of $A_m$}, Math. Ann., 256 (1981), 401-418.
\label{ADK}  

\bibitem[BCOQ]{BCOQ} F.~Baccelli, G.~Cohen, G.J.~Olsder and J.-P.~Quadrat:
{\em Synchronization and Linearity: An Algebra for Discrete Event Systems},
Wiley Interscience, 1992.

  
\bibitem[Bu]{Bu} P. Butkovic: \emph{Max-linear Systems: Theory and Algorithms}, Springer Monographs in Mathematics, Springer-Verlag (2010).
\label{Bu}

 
\bibitem[DS]{DS} M. Develin and B. Sturmfels, \emph{Tropical Convexity}, Documenta Mathematica {\bf 9} (2004) 1-27.
\label{DS}

\bibitem[Ge]{Ge} M. Gerstenhaber:  \emph{On dominance and varieties of commuting matrices}, Annals of Mathematics 73 (1961), 324-348.
\label{Ge}

\bibitem[GS]{GS} R.M. Guralnick and B.A. Sethuraman:  \emph{Commuting pairs and triples of matrices and related varieties}, Linear Algebra Appl., 310 (2000), 139-148.
\label{GS}

\bibitem[Je]{Je} Jensen, A.: \emph{{G}fan, a software system for {G}r\"{o}bner fans and tropical varieties}, available at \url{http://home.imf.au.dk/jensen/software/gfan/gfan.html}.

\bibitem[JK]{JK}  Joswig, M. and Kulas, K: \emph{Tropical and ordinary convexity combined},
Advances in geometry, 10 (2010) 333-352.


\bibitem[Ki]{Ki} J. Krivulin:  \emph{ A solution of a tropical linear vector equation}, (2012), Recent Advances in Computer Engineering Series, 5 (2012), 244-249
\label{LP}

\bibitem[Kn]{Kn} A. Knutson: \emph{Some schemes related to the commuting variety}, J. Algebraic Geom. 14 (2005), 283-294.
\label{Kn}

\bibitem[KSS]{KSS} R. D. Katz, H. Schneider and S. Sergeev: \emph{On commuting matrices in max algebra and in classical nonnegative algebra}, Linear Algebra Appl., 436 (2012), 276-292
  
\bibitem[LP]{LP} J. Linde and M. Puente:  \emph{Commuting normal idempotent tropical matrices: an algebraic-geometric approach}, (2012), ArXiv 1209.0660
\label{LP}

\bibitem[MS]{MS} D. Maclagan and B. Sturmfels: \emph{Introduction to tropical geometry}, Graduate Studies in Mathematics vol. 161, AMS, Providence RI (2015)

\bibitem[MT]{MT}T. Motzkin and O. Taussky-Todd:  \emph{Pairs of matrices with property L. II}, Transactions of the AMS {\bf 80} (1955), 387-401. 
\label{MT}

\bibitem[OCV]{OCV} K. C. O'Meara , J. Clark and C. I Vinsonhaler: \emph{Advanced topics in linear algebra}, Oxford University Press, Oxford (2011) 

\bibitem[SSB]{SSB}  S. Sergeev, H. Schneider, and P. Butkovi{\v{c}}: \emph{On visualization scaling, subeigenvectors and Kleene stars in max algebra}, Linear Algebra Appl., 431 (2009), 2395-2406.

\bibitem[Tr]{Tr} N. M. Tran: \emph{Pairwise ranking: choice of method can produce arbitrarily different rank order}, 438 (2013), 1012-1024.


\bibitem[Wi]{Wi} S. Willerton:  \emph{Tight spans, Isbell completions and semi-tropical modules},
Theory and Applications of Categories 28 (2013) 696-732.

\bibitem[BK]{BK} E. Baldwin and P. Klemperer: \emph{Tropical geometry to analyse demand}, Technical Report, (2012)


\end{thebibliography}
\end{document}